\documentclass[12pt]{amsart}

\usepackage{amssymb}
\usepackage{verbatim}
\usepackage[toc,page]{appendix}
\usepackage{mathrsfs}

\newtheorem{thm}{Theorem}[section]

\newtheorem{lem}[thm]{Lemma}
\newtheorem{cor}[thm]{Corollary}

\theoremstyle{definition}

\theoremstyle{remark}

\newtheorem*{rem}{Remark}

\numberwithin{equation}{section}

\newcommand{\C}{\mathbf{C}}

\newcommand{\Mod}[1]{\ (\textup{mod}\ #1)}

\providecommand{\sgn}{\operatorname{sgn}}

\providecommand{\sgn}{\operatorname{sgn}}
\providecommand{\sym}{\operatorname{sym}}

\DeclareMathOperator{\arcsinh}{arcsinh}
\DeclareMathOperator{\artanh}{artanh}

\begin{document}

\title[On the mean value of symmetric square $L$-functions]{On the mean value of symmetric square $L$-functions}

\author{Olga  Balkanova}
\address{University of Turku, Department of Mathematics and Statistics, Turku, 20014, Finland}
\email{olgabalkanova@gmail.com}
\thanks{Research of the first author is supported by Academy of Finland project no. $293876$.}

\author{Dmitry  Frolenkov}
\address{Steklov Mathematical Institute of Russian Academy of Sciences,  8 Gubkina st., Moscow, 119991, Russia}
\email{frolenkov@mi.ras.ru}
\begin{abstract}
This paper studies the first moment of symmetric-square $L$-functions at the critical point  in the weight aspect. Asymptotics with the best known error term $O(k^{-1/2})$ were obtained independently by Fomenko in 2005 and by Sun in 2013.
We prove that there is an extra main term of size $k^{-1/2}$ in the asymptotic formula and show that the remainder term decays exponentially in $k$. The twisted first moment was evaluated asymptotically by Ng Ming Ho with the error bounded by $lk^{-1/2+\epsilon}$. We improve the error bound to $l^{5/6+\epsilon}k^{-1/2+\epsilon}$ unconditionally and to $l^{1/2+\epsilon}k^{-1/2}$ under the Lindel\"{o}f hypothesis for quadratic Dirichlet $L$-functions. 
\end{abstract}


\maketitle

\tableofcontents


\section{Introduction}

Asymptotic behavior of high moments of $L$-functions within different families can be predicted using random matrix theory \cite{CFKRS} or multiple Dirichlet series \cite{DGH}. However, obtaining asymptotic formulas with sharp error bounds is a hard problem even in case of small moments.

One of the most challenging families is symmetric square $L$-functions in weight aspect. Gelbart and Jacquet \cite{GJ} proved that these are $L$-function attached to $GL(3)$ cusp forms.

Despite numerous efforts, even an upper bound for the second moment of symmetric square $L$-functions remains an open problem. See Conjecture $1.2$ \cite{K2}. 

The first moment  has been studied intensively during the last decades. See \cite{F,K,KS,Lau,NMH,Sun}. Nethertheless, even the best known asymptotic error estimates do not appear to be sharp.

The present paper aims to optimize error bounds in existing asymptotic formulas. 
With this goal, we  prove an exact formula for the twisted first moment of symmetric square $L$-functions, and apply  the Liouville-Green method (also called WKB approximation) to estimate remainder terms. This technique, originating from the theory of approximation of second-order differential equations, is quite unusual for analytic number theory, yet very effective. See, for example, \cite{BF, Zav}.

\section{Main results}
Let $S_{2k}(1)$ denote the space of holomorphic cusp forms of weight $2k\geq 2$ with respect to the full modular group. Denote by $H_{2k}$ the normalized Hecke basis for $S_{2k}(1)$. Every $f \in H_{2k}$ has a Fourier expansion of the form
\begin{equation}
f(z)=\sum_{n\geq 1}\lambda_f(n)n^{k-1/2}\exp(2\pi inz),
\end{equation}
\begin{equation}
\lambda_f(1)=1.
\end{equation}
For $\Re{s}>1$ the associated symmetric square $L$-function is given by
\begin{equation}
L(\sym^2f,s)=\zeta(2s)\sum_{n=1}^{\infty}\frac{\lambda_f(n^2)}{n^s}.
\end{equation}
Let $\Gamma(s)$ be the Gamma function and
\begin{equation}
L_{\infty}(s):=\pi^{-3s/2}\Gamma\left(\frac{s+1}{2}\right)\Gamma\left(\frac{s+k-1}{2}\right)\Gamma\left(\frac{s+k}{2}\right).
\end{equation}
Shimura \cite{S} showed that the completed $L$-function
\begin{equation*}
\Lambda(\sym^2f,s):=L_{\infty}(s)L(\sym^2f,s)
\end{equation*}
is entire and satisfies the functional equation
\begin{equation}
\Lambda(\sym^2f,s)=\Lambda(\sym^2f,1-s).
\end{equation}

Consider
\begin{equation}
M_1(l,s):=\sum_{f \in H_{2k}}^{h}\lambda_f(l^2)L(\sym^2f,s).
\end{equation}
The subscript $h$ in the formula above indicates that the expression in the sum is multiplied by the harmonic weight
$\Gamma(2k-1)/((4\pi)^{2k-1} \langle f,f\rangle_1)$, where $\langle f,f\rangle_1$ is the Petersson inner product on the space of level $1$ holomorphic modular forms.

Denote by $\gamma$ the Euler constant and by
$\psi(s)$ logarithmic derivative of the Gamma function.
Let ${}_2F_{1}(a,b,c;x)$ be the Gauss hypergeometric function and
\begin{equation}\label{defphi}
\Phi_k(x):=\frac{\Gamma(k-1/4)\Gamma(3/4-k)}{\Gamma(1/2)}{}_2F_{1}\left(k-\frac{1}{4},\frac{3}{4}-k,1/2;x \right),
\end{equation}
\begin{equation}
\Psi_k(x):=x^k\frac{\Gamma(k-1/4)\Gamma(k+1/4)}{\Gamma(2k)}{}_2F_{1}\left(k-\frac{1}{4},k+\frac{1}{4},2k;x \right).
\end{equation}
We prove the following exact formula for the twisted first moment.
\begin{thm}\label{thm:explicitformula}
For any $l \geq 1$ one has
\begin{multline}\label{mainformula}
M_1(l,1/2)=\frac{1}{2\sqrt{l}}\biggl(-2\log{l}-3\log{2\pi}+\frac{\pi}{2}+4\gamma+\psi(1)+\\ \psi(k-1/4)+ \psi(k+1/4) \biggr)+
\frac{\sqrt{2\pi}(-1)^k}{2\sqrt{l}}\frac{\Gamma(k-1/4)}{\Gamma(k+1/4)}\mathscr{L}_{-4l^2}(1/2)+\\
\frac{1}{\sqrt{l}} \sum_{1\leq n<2l}
\mathscr{L}_{n^2-4l^2}(1/2)\Phi_k\left(\frac{n^2}{4l^2}\right)+
\frac{1}{l\sqrt{2}} \sum_{n>2l}\mathscr{L}_{n^2-4l^2}(1/2)\sqrt{n}\Psi_k\left( \frac{4l^2}{n^2}\right),
\end{multline}
where
\begin{equation}\label{Lbyk}
\mathscr{L}_n(s)=\frac{\zeta(2s)}{\zeta(s)}\sum_{q=1}^{\infty}\frac{1}{q^s}\left( \sum_{1\leq t \leq 2q;t^2 \equiv n \Mod{4q}}1\right).
\end{equation}
\end{thm}

A similar formula, where the last two summands are expressed in terms of Legendre function of the first kind, was established by a different method by Zagier \cite[Theorem~1]{Z}. Zagier's formula was applied by Kohnen and Sengupta \cite{KS} to prove an upper bound for $M_1(1,1/2)$, by Fomenko \cite{F} to obtain an asymptotic formula for $M_1(1,1/2)$, and by Luo \cite{L} to estimate the second moment of $L(\sym^2f,1/2)$ over short intervals.

The proof of Theorem \ref{thm:explicitformula} is quite simple and make use of Petersson's trace formula and the functional equation for the Lerch zeta function. 

When $l=1$, exact formula \eqref{mainformula} allows isolating the second main term of size $k^{-1/2}$ in the asymptotic formula so that the remainder term decays exponentially.

\begin{cor}\label{thm:asympformula2}
For some $c>0$ one has
\begin{multline}\label{mainformula3}
M_1(1,1/2)=\frac{1}{2}\biggl(-3\log{2\pi}+\frac{\pi}{2}+4\gamma+\psi(1)+\\ \psi(k-1/4)+ \psi(k+1/4) \biggr)+
\frac{\sqrt{2\pi}(-1)^k}{2}\frac{\Gamma(k-1/4)}{\Gamma(k+1/4)}L(1/2,\chi_{-4})+\\ \Phi_k\left(\frac{1}{4}\right)L(1/2,\chi_{-3})+O\left(\frac{1}{\sqrt{k}}exp(-ck)\right),
\end{multline}
where $L(1/2,\chi_D)$ is a Dirichlet $L$-function for primitive quadratic character of conductor $D$.
\end{cor}
\begin{rem}
After posting the first version of this paper to the arXiv, the authors have been informed by Shenhui Liu that
he has independently obtained an asymptotic formula similar to \eqref{mainformula3} by using an approximate functional equation. See \cite{Liu}.
\end{rem}

Corollary \ref{thm:asympformula2} improves the series of previously known results with the following error bounds:
\begin{itemize}
\item $k^{-0.008}$ proved by Lau \cite{Lau} in 2002;
\item $k^{-1/20}$ proved by  Khan \cite{K} in 2007;
\item $k^{-1/2}$ proved by Fomenko \cite{F} in 2005 and by Sun \cite{Sun} in 2013.
\end{itemize}

\begin{cor}\label{thm:asympformula}
For any $\epsilon>0$, $l >1$ one has
\begin{multline}\label{mainformula2}
M_1(l,1/2)=\frac{1}{2\sqrt{l}}\biggl(-2\log{l}-3\log{2\pi}+\frac{\pi}{2}+4\gamma+\psi(1)+\\ \psi(k-1/4)+ \psi(k+1/4) \biggr)+O\left(\frac{l^{5/6+\epsilon}}{\sqrt{k}}\right).
\end{multline}
Assuming the Lindel\"{o}f hypothesis for quadratic Dirichlet $L$-functions, the error term above can be replaced by
$O\left(l^{1/2+\epsilon}k^{-1/2}\right).$

\end{cor}
This improves the error bound $lk^{-1/2+\epsilon}$ proved by Ng Ming Ho \cite[Theorem~2.1.1]{NMH}.

\section{Notations and tools}

Let $e(x)=exp(2\pi ix)$. For $ v \in \C $ let
\begin{equation}
\tau_v(n)=\sum_{n_1n_2=n}\left( \frac{n_1}{n_2}\right)^v.
\end{equation}
The classical Kloosterman sum is defined by
\begin{equation*}
S(n,m;c)=\sum_{\substack{a\Mod{c}\\ (a,c)=1}}e\left( \frac{an+a^*m}{c}\right), \quad aa^*\equiv 1\Mod{c}.
\end{equation*}
\begin{lem}(Weil's bound  \cite{W}) One has
\begin{equation}
|S(m,n;c)|\leq \tau_0(c)\sqrt{(m,n,c)}\sqrt{c}.
\end{equation}
\end{lem}
Let $J_{v}(x)$ be the Bessel function of the first kind.
\begin{lem}(Petersson's trace formula, \cite{P}) For $2k \geq 12$  and integral $l,n \geq 1$ one has
\begin{equation}\label{Pet}
\sum_{f \in H_{2k}}^{h}\lambda_f(l)\lambda_f(n)=\delta_{l,n}+2\pi i^{2k}\sum_{c =1}^{\infty}\frac{S(l,n;c)}{c}J_{2k-1}\left( \frac{4\pi \sqrt{ln}}{c}\right).
\end{equation}
\end{lem}

The Lerch zeta function
\begin{equation}
\zeta(\alpha,\beta,s)=\sum_{n+\alpha>0}\frac{e(n\beta)}{(n+\alpha)^s}
\end{equation}
was introduced by Lipschitz \cite{Lip} in $1857$ and was named after Lerch, who proved in $1887$ the following functional equation.
\begin{lem}(\cite{Ler})
One has
\begin{equation}\label{LerchFE}
\zeta(\alpha,0,s)=\frac{\Gamma(1-s)}{(2\pi)^{1-s}}\biggl(-ie\left(\frac{s}{4}\right)\zeta(0,\alpha,1-s)+
ie\left(-\frac{s}{4}\right)\zeta(0,-\alpha,1-s)\biggr).
\end{equation}
\end{lem}

\section{Some properties of $\mathscr{L}_{n}(s)$}

The main references for this section are \cite{B,SY,Z}.
Function \eqref{Lbyk} can be written as follows
\begin{equation}
\mathscr{L}_{n}(s)=\frac{\zeta(2s)}{\zeta(s)}\sum_{q=1}^{\infty}\frac{\rho_q(n)}{q^{s}}=\sum_{q=1}^{\infty}\frac{\lambda_q(n)}{q^{s}},
\end{equation}
where
\begin{equation}
\rho_q(n):=\#\{x\Mod{2q}:x^2\equiv n\Mod{4q}\},
\end{equation}
\begin{equation}
\lambda_q(n):=\sum_{q_{1}^{2}q_2q_3=q}\mu(q_2)\rho_{q_3}(n).
\end{equation}
For a fixed $n$ both $\rho_q(n)$ and $\lambda_q(n)$ are multiplicative functions of $q$. Furthermore,
for $n \equiv 2,3 \Mod{4}$ the function $\rho_q(n)$ is identically zero. Therefore, $\mathscr{L}_n(s)$ does not vanish only for $n \equiv 0,1 \Mod{4}.$ If $n=0$ then
\begin{equation}
\mathscr{L}_{n}(s)=\zeta(2s-1).
\end{equation}
Otherwise, for $n=Dl^2$ with $D$ fundamental discriminant we have
\begin{equation}\label{ldecomp}
\mathscr{L}_{n}(s)=l^{1/2-s}T_{l}^{(D)}(s)L(s,\chi_D),
\end{equation}
where $L(s,\chi_D)$ is a Dirichlet L-function for primitive quadratic character $\chi_D$ and
\begin{equation}\label{eq:td}
T_{l}^{(D)}(s)=\sum_{l_1l_2=l}\chi_D(l_1)\frac{\mu(l_1)}{\sqrt{l_1}}\tau_{s-1/2}(l_2).
\end{equation}

The completed $L$-function
\begin{equation}
\mathscr{L}_{n}^{*}(s)=(\pi/|n|)^{-s/2}\Gamma(s/2+1/4-\sgn{n}/4)\mathscr{L}_{n}(s)
\end{equation}
satisfies the functional equation
\begin{equation}\label{functlstar}
\mathscr{L}_{n}^{*}(s)=\mathscr{L}_{n}^{*}(1-s).
\end{equation}
\begin{lem}
One has
\begin{equation}\label{eq:sumofklsums}
\sum_{q=1}^{\infty}\frac{1}{q^{1+s}}\sum_{c\Mod{q}}S(l^2,c^2;q)e\left(\frac{nc}{q}\right)=\frac{1}{\zeta(2s)}
\mathscr{L}_{n^2-4l^2}(s).
\end{equation}
\end{lem}
\begin{proof}
Consider
\begin{multline*}
S:=\sum_{c\Mod{q}}S(l^2,c^2;q)e\left(n\frac{c}{q}\right)=\\
\sum_{c\Mod{q}}\sum_{\substack{a \Mod{q}\\(a,q)=1}}e\left( \frac{ac^2+a^*l^2+nc}{q}\right),
\end{multline*}
where $aa^*\equiv 1\Mod{q}$. Making the change of variables $c=c_1a^*$, we have
\begin{multline*}
S=
\sum_{\substack{a \Mod{q}\\(a,q)=1}}\sum_{c_1 \Mod{q}}e\left(\frac{c_{1}^{2}a^*+l^2a^*+nc_1a^*}{q}\right)=\\
\sum_{c_1 \Mod{q}}S(0,c_{1}^{2}+l^2+nc_1;q)=\sum_{c\Mod{q}}\sum_{\substack{bd=q\\b|c^2+l^2+nc}}\mu(d)b=\\
\sum_{bd=q}\mu(d)b\sum_{\substack{c \Mod{q}\\ c^2+l^2+nc\equiv 0 \Mod{b}}}1=\sum_{bd=q}\mu(d)b
\sum_{\substack{c\Mod{b}\\ c^2+l^2+nc\equiv 0\Mod{b}}}\frac{q}{b}.
\end{multline*}
The condition $$c^2+l^2+nc\equiv 0 \Mod{b}$$ is equivalent to $$(2c+n)^2+4l^2-n^2\equiv 0\Mod{4b}.$$
Hence
\begin{equation*}
S=q\sum_{bd=q}\mu(d)\sum_{\substack{c\Mod{2b}\\ c^2\equiv n^2-4l^2\Mod{4b}}}1=q\sum_{bd=q}\mu(d)\rho_b(n^2-4l^2).
\end{equation*}
Consequently,
\begin{multline*}
\sum_{q=1}^{\infty}\frac{1}{q^{1+s}}\sum_{c\Mod{q}}S(l^2,c^2;q)e\left(\frac{nc}{q}\right)=
\sum_{q=1}^{\infty}\frac{1}{q^s}\sum_{bd=q}\mu(d)\rho_b(n^2-4l^2)=\\
\sum_{b=1}^{\infty}\frac{\rho_b(n^2-4l^2)}{b^s}\sum_{q=1}^{\infty}\frac{\mu(q)}{q^s}=
\frac{1}{\zeta(s)}\sum_{q=1}^{\infty}\frac{\rho_q(n^2-4l^2)}{q^s}=\frac{\mathscr{L}_{n^2-4l^2}(s)}{\zeta(2s)}
.
\end{multline*}
\end{proof}

\begin{lem}\label{lem:subconvexity}
Assume that $d\neq 0$. For any $\epsilon>0$ one has
\begin{equation}\label{eq:subconvexity}
\mathscr{L}_d(1/2)\ll d^{1/6+\epsilon}.
\end{equation}
If the Lindel\"{o}f hypothesis for Dirichlet $L$-functions is true, then
\begin{equation}\label{eq:lindelof}
\mathscr{L}_d(1/2)\ll d^{\epsilon}.
\end{equation}
\end{lem}
\begin{proof}
For $d=Dl^2$ with $D$ fundamental discriminant one has
\begin{equation*}
\mathscr{L}_{d}(1/2)=T_{l}^{(D)}(1/2)L(1/2,\chi_D)
\end{equation*}
by equation \eqref{ldecomp}.
It follows from equality \eqref{eq:td} that
\begin{equation*}
T_{l}^{(D)}(1/2)\ll \sum_{l_1l_2=l}\frac{\sigma(l_2)}{\sqrt{l_1}}\ll \sum_{l_1|l}\frac{(l/l_1)^{\epsilon}}{\sqrt{l_1}}\ll l^{\epsilon}.
\end{equation*}
By \cite[Corollary~1.5]{CI} for any $\epsilon>0$ one has
\begin{equation*}
L(1/2, \chi_D)\ll D^{1/6+\epsilon}.
\end{equation*}
This implies the required bounds for the function $\mathscr{L}_{d}(1/2)$.
\end{proof}

\section{Exact formula}


\begin{lem} For $\Re{s}>3/2$ one has
\begin{multline}\label{eq:M1ls}
M_1(l,s)=\frac{\zeta(2s)}{l^{s}}+
\frac{(2\pi)^{s}i^{2k}}{2l^{1-s}}\frac{\Gamma(k-s/2)}{\Gamma(k+s/2)}\mathscr{L}_{-4l^2}(s)+\\
(2\pi)^{s}i^{2k}\sum_{n=1}^{\infty}\frac{1}{n^{1-s}}\mathscr{L}_{n^2-4l^2}(s)I\left( \frac{n}{l}\right),
\end{multline}
where 
\begin{equation}\label{eq:integralI}
I(x):=\frac{1}{2\pi i}\int_{(\Delta)}\frac{\Gamma(k-1/2+t/2)}{\Gamma(k+1/2-t/2)}\Gamma(1-s-t)\sin\left( \frac{s+t}{2}\right)x^tdt
\end{equation}
with $1-2k<\Delta<1-\Re{s}.$
\end{lem}
\begin{proof}
By the Petersson trace formula
\begin{multline*}
M_1(l,s)=\zeta(2s)\sum_{n=1}^{\infty}\frac{1}{n^{s}}\sum_{f \in H_{2k}}^{h}\lambda_f(l^2)\lambda_f(n^2)=\\
\frac{\zeta(2s)}{l^{s}}+2\pi i^{2k}\zeta(2s)\sum_{q=1}^{\infty}\frac{1}{q}
\sum_{n=1}^{\infty}\frac{S(l^2,n^2;q)}{n^{s}}J_{2k-1}\left(4\pi \frac{ln}{q}\right).
\end{multline*}

The change of order of summation above is justified by the absolute convergence for $\Re{s}>3/2$, which follows from the standard estimates
\begin{equation*}
S(l^2,n^2;q)\ll q^{1/2+\epsilon}(l^2,n^2,q)^{1/2} \text{ for any } \epsilon>0
\end{equation*}
and
\begin{equation*}
J_{2k-1}\left(4\pi \frac{ln}{q}\right)\ll \begin{cases}
(ln/q)^{2k-1} & q>ln\\
(ln/q)^{-1/2} & q<ln.
\end{cases}
\end{equation*}
Next, we use the Mellin-Barnes representation for the Bessel function
\begin{multline*}
M_1(l,s)=\frac{\zeta(2s)}{l^{s}}+2\pi i^{2k}\zeta(2s)\sum_{q=1}^{\infty}\frac{1}{q}\times 
\\ \frac{1}{4\pi i}
\int_{(\Delta)}\frac{\Gamma(k-1/2+t/2)}{\Gamma(k+1/2-t/2)}\sum_{n=1}^{\infty}\frac{S(l^2,n^2;q)}{n^{t+s}}
\left(\frac{q}{2\pi l} \right)^tdt,
\end{multline*}
where $1-2k<\Delta<0$.
To guarantee the absolute convergence of the integral over $t$ and the sums over $q,n$ we require that
$$\max(1-2k,1-\Re{s})<\Delta<-1/2,$$ which is true for $\Re{s}>3/2$.
Consider
\begin{multline*}
\sum_{n=1}^{\infty}\frac{S(l^2,n^2;q)}{n^{t+s}}=\sum_{c\Mod{q}}\sum_{n \equiv c\Mod{q}}\frac{S(l^2,c^2;q)}{n^{t+s}}=\\ \sum_{c\Mod{q}}S(l^2,c^2;q) \sum_{n=1}^{\infty}\frac{1}{(c+nq)^{t+s}}=
\sum_{c\Mod{q}}\frac{S(l^2,c^2;q)}{q^{t+s}}\zeta\left(\frac{c}{q},0,t+s\right).
\end{multline*}
Note that the Lerch zeta function has a pole at $t=1-s$ for all $c$.
The next step is to apply functional equation \eqref{LerchFE} for the Lerch zeta function  which is only possible when $\Re(s+t)<0$. Accordingly, we move the $t$-contour to the left up to $\Delta_1:=-s-\epsilon$, crossing a simple
pole at $t=1-s$. Therefore,
\begin{multline*}
M_1(l,s)=\frac{\zeta(2s)}{l^s}+2\pi i^{2k}\zeta(2s)\frac{\Gamma(k-s/2)}{\Gamma(k+s/2)}\sum_{q=1}^{\infty}
\sum_{c\Mod{q}}\frac{S(l^2,c^2;q)}{2q^2}\times \\ \left( \frac{q}{2\pi l}\right)^{1-s}+2\pi i^{2k}\zeta(2s)
\sum_{q=1}^{\infty}\frac{1}{q} \frac{1}{4\pi i}\times \\\int_{(\Delta_1)}\frac{\Gamma(k-1/2+t/2)}{\Gamma(k+1/2-t/2)}
\left(\frac{q}{2\pi l} \right)^t\sum_{c \Mod{q}}\frac{S(l^2,c^2;q)}{q^{s+t}}\zeta(c/q,0;s+t)ds.
\end{multline*}
Using functional equation \eqref{LerchFE}, we obtain
\begin{multline*}
\sum_{c \Mod{q}}S(l^2,c^2;q)\zeta(c/q,0;s+t)=2(2\pi)^{s+t-1}\Gamma(1-s-t)\times \\\sin\left(\pi \frac{s+t}{2}\right)
\sum_{c\Mod{q}}S(l^2,c^2;q)\zeta(0,c/q;1-s-t).
\end{multline*}
Substituting this into $M_1(l,s)$ and opening the Lerch zeta function, one has
\begin{multline*}
M_1(l,s)=\frac{\zeta(2s)}{l^s}+\frac{(2\pi)^s i^{2k}}{2l^{1-s}}\zeta(2s)\frac{\Gamma(k-s/2)}{\Gamma(k+s/2)}
\sum_{q=1}^{\infty}\sum_{c\Mod{q}}\frac{S(l^2,c^2;q)}{q^{1+s}}+\\
(2\pi)^si^{2k}\zeta(2s)\sum_{n=1}^{\infty}\frac{1}{n^{1-s}}\sum_{q=1}^{\infty}\frac{1}{q^{1+s}}
\sum_{c\Mod{q}}S(l^2,c^2;q)e\left(n\frac{c}{q}\right)I\left(\frac{n}{l} \right),
\end{multline*}
where $I(x)$ is defined by equation \eqref{eq:integralI}.
Finally, computing the sums over $c$ and $q$ using formula \eqref{eq:sumofklsums}, we prove the Lemma.
\end{proof}

\begin{lem}
If $x \geq 2$, then 
\begin{multline}\label{integralIgeq2}
I(x)=\frac{2^{2k}(-1)^k}{2^s\sqrt{\pi}}\cos\left( \frac{\pi s}{2}\right)x^{1-2k}\frac{\Gamma(k-s/2)\Gamma(k+1/2-s/2)}{\Gamma(2k)}\times \\ {}_2F_{1}\left(k-s/2,k+1/2-s/2,2k;\frac{4}{x^2} \right).
\end{multline}
\end{lem}
\begin{proof}
Moving the contour of integration in \eqref{eq:integralI} to the left, we cross simple poles at $t=1-2k-2j$, $j=0,1,2,\ldots$
Therefore,
\begin{equation*}
I(x)=2(-1)^k\cos \left(\frac{\pi s}{2} \right)x^{1-2k}\sum_{j=0}^{\infty}\frac{1}{j!}\frac{\Gamma(2k-s+2j)}{\Gamma(2k+j)}x^{-2j}.
\end{equation*}
By \cite[Eq.~5.5.5]{HMF} we have
\begin{equation*}
\Gamma(2(k+j-s/2))=\frac{2^{2k-1-s+2j}}{\sqrt{\pi}}\Gamma(k+j-s/2)\Gamma(k+j+1/2-s/2).
\end{equation*}
This yields
\begin{multline*}
I(x)=\frac{2(-1)^k}{\sqrt{\pi}}2^{2k-1-s}\cos \left(\frac{\pi s}{2} \right)x^{1-2k}\sum_{j=0}^{\infty}\frac{1}{j!}\frac{\Gamma(k-s/2+j)}{\Gamma(2k+j)}\times\\ \Gamma(k+1/2-s/2+j)\left(\frac{4}{x^2} \right)^j=
\frac{2^{2k}(-1)^k}{2^s\sqrt{\pi}}\cos\left( \frac{\pi s}{2}\right)x^{1-2k}\times \\ \frac{\Gamma(k-s/2)\Gamma(k+1/2-s/2)}{\Gamma(2k)}{}_2F_{1}\left(k-s/2,k+1/2-s/2,2k;\frac{4}{x^2} \right).
\end{multline*}
\end{proof}

\begin{lem}
One has
\begin{equation}\label{integralIeq2}
I(2)=\frac{2(-1)^k}{2^s\sqrt{\pi}}\cos{\left(\frac{\pi s}{2}\right)}\frac{\Gamma(k-s/2)\Gamma(k+1/2-s/2)}{\Gamma(k+s/2)\Gamma(k-1/2+s/2)}\Gamma(s-1/2).
\end{equation}
\end{lem}
\begin{proof}
Letting $x=2$ in \eqref{integralIgeq2} and applying  \cite[Eq.~15.4.20]{HMF}, we find
\begin{equation*}
{}_2F_{1}\left(k-s/2,k+1/2-s/2,2k;1 \right)=\frac{\Gamma(2k)\Gamma(s-1/2)}{\Gamma(k+s/2)\Gamma(k-1/2+s/2)}.
\end{equation*}
The assertion follows.
\end{proof}

\begin{lem}
If $x<2$, then 
\begin{multline}\label{integralIl2}
I(x)=\frac{(-1)^k}{\sqrt{\pi}}\sin\left( \frac{\pi s}{2}\right)x^{1-s}\frac{\Gamma(k-s/2)\Gamma(1-k-s/2)}{\Gamma(1/2)}\times \\
{}_2F_{1}\left( k-\frac{s}{2},1-k-\frac{s}{2},1/2;\frac{x^2}{4}\right).
\end{multline}
\end{lem}
\begin{proof}
Moving the contour of integration in \eqref{eq:integralI} to the right we cross simple poles at $t=1-s+j$, $j=0,1,2,\ldots$ Accordingly,
\begin{equation*}
I(x)=\sum_{j=0}^{\infty}\frac{(-1)^j}{j!}\frac{\Gamma(k-s/2+j/2)}{\Gamma(k+s/2-j/2)}\sin \left(\pi \frac{1+j}{2}\right)x^{1-s+j}.
\end{equation*}
Note that
\begin{equation*}
\sin \left(\pi \frac{1+j}{2}\right)=\cos \left( \frac{\pi j}{2}\right)=\begin{cases}
0 & j \text{ is odd,}\\
(-1)^m & j=2m.
\end{cases}
\end{equation*}
Thus
\begin{equation*}
I(x)=\sum_{m=0}^{\infty}\frac{1}{(2m)!}\frac{\Gamma(k-s/2+m)}{\Gamma(k+s/2-m)}(-1)^m x^{1-s+2m}.
\end{equation*}
In order to express $I(x)$ in terms of the Gauss hypergeometric function we apply \cite[Eq.~5.5.5]{HMF}, obtaining
\begin{equation*}
(2m)!=\Gamma(2(m+1/2))=\frac{1}{\sqrt{\pi}}2^{2m}\Gamma(m+1/2)\Gamma(m+1).
\end{equation*}
Furthermore, by Euler's reflection formula
\begin{equation*}
\Gamma(k+s/2-m)=\frac{\pi}{(-1)^{k-m}\sin{(\pi s/2)}\Gamma(1-k-s/2+m)}.
\end{equation*}
Finally,
\begin{multline*}
I(x)=\frac{(-1)^k}{\sqrt{\pi}}\sin\left( \frac{\pi s}{2}\right)x^{1-s}
\sum_{m=0}^{\infty}\frac{1}{m!}\frac{\Gamma(k-s/2+m)}{\Gamma(m+1/2)}\times 
\\ \Gamma(1-k-s/2+m)\left(\frac{x^2}{4}\right)^{m}=
 \frac{(-1)^k}{\sqrt{\pi}}\sin\left( \frac{\pi s}{2}\right)x^{1-s}
\times \\ \frac{\Gamma(k-s/2)\Gamma(1-k-s/2)}{\Gamma(1/2)}
{}_2F_{1}\left( k-\frac{s}{2},1-k-\frac{s}{2},1/2;\frac{x^2}{4}\right).
\end{multline*}
\end{proof}

Next, we  substitute  equations \eqref{integralIgeq2}, \eqref{integralIeq2}, \eqref{integralIl2} into
expression \eqref{eq:M1ls}, proving the exact formula for the shifted first moment.
\begin{thm}
For any $l \geq 1$ and $2-2k<\Re{s}<2k-1$ one has
\begin{multline}\label{eq:M1ls2}
M_1(l,s)=\frac{\zeta(2s)}{l^{s}}+
\frac{(2\pi)^{s}i^{2k}}{2l^{1-s}}\frac{\Gamma(k-s/2)}{\Gamma(k+s/2)}\mathscr{L}_{-4l^2}(s)+\\
\frac{(2\pi)^s}{\sqrt{\pi}}\frac{\zeta(2s-1)}{l^{1-s}}\cos\left(\frac{\pi s}{2}\right)\frac{\Gamma(k-s/2)\Gamma(k+1/2-s/2)}{\Gamma(k+s/2)\Gamma(k-1/2+s/2)}\Gamma(s-1/2)+\\
\frac{(2\pi)^{s}\sin(\pi s/2)}{\sqrt{\pi}l^{1-s}}\sum_{1\leq n<2l}\mathscr{L}_{n^2-4l^2}(s)
\frac{\Gamma(k-s/2)\Gamma(1-k-s/2)}{\Gamma(1/2)}\times\\
{}_2F_1\left( k-\frac{s}{2},1-k-\frac{s}{2},1/2;\left(\frac{n}{2l}\right)^2\right)+
\frac{2^{2k}\pi^{s}\cos(\pi s/2)}{\sqrt{\pi}}\times \\\sum_{n>2l}\mathscr{L}_{n^2-4l^2}(s)
\frac{\Gamma(k-s/2)\Gamma(k+1/2-s/2)}{\Gamma(2k)} \times \\
\frac{1}{n^{1-s}}\left(\frac{n}{l}\right)^{1-2k}
{}_2F_1\left( k-\frac{s}{2},k+1/2-\frac{s}{2},2k;\left( \frac{2l}{n}\right)^2\right).
\end{multline}
\end{thm}

Note that in equation \eqref{eq:M1ls2} only the first and the third summands have poles at $s=1/2$.
Computing the limit as $s \rightarrow 1/2$ we find that these poles cancel each other. This allows proving the exact formula for the first moment of symmetric square $L$-functions at the critical point given by Theorem 
\ref{thm:explicitformula}.

\section{Liouville-Green approximation of hypergeometric functions}
To estimate the functions  $\Phi_k(x)$ and $\Psi_k(x)$, appearing in exact formula \eqref{mainformula}, we apply the Liouville-Green method. It turns out that these special functions have a similar behavior with the ones occurring in the exact formula for the second moment of cusp form $L$-functions. See \cite[Theorem~4.2]{BF}.
\subsection{Properties of $\Phi_k$}

Consider the function $\Phi_k(x)$ for $0<x<1$.
\begin{lem} One has
\begin{multline}\label{eq:decompphi}
\Phi_k(x)=-\pi\Biggl({}_2F_1\left( 2k-\frac{1}{2}, \frac{3}{2}-2k,1;\frac{1-\sqrt{x}}{2}\right)
+\\{}_2F_1\left( 2k-\frac{1}{2}, \frac{3}{2}-2k,1;\frac{1+\sqrt{x}}{2}\right)\Biggr).
\end{multline}
\end{lem}
\begin{proof}
Applying the quadratic transformation given by \cite[Eq.~15.8.27]{HMF} we obtain
\begin{multline*}
\Phi_k(x)=\Gamma(k-1/4)\Gamma(3/4-k)\Gamma(k+1/4)\Gamma(5/4-k)\times
\\ \frac{1}{2\Gamma^2(1/2)}\Biggl({}_2F_1\left( 2k-\frac{1}{2}, \frac{3}{2}-2k,1;\frac{1-\sqrt{x}}{2}\right)
+\\{}_2F_1\left( 2k-\frac{1}{2}, \frac{3}{2}-2k,1;\frac{1+\sqrt{x}}{2}\right)\Biggr).
\end{multline*}
Note that $\Gamma(1/2)=\sqrt{\pi}$. 
Euler's reflection formula yields
\begin{equation*}
\Gamma(k-1/4)\Gamma(k+1/4)\Gamma(3/4-k)\Gamma(5/4-k)=-2\pi^2.
\end{equation*}
The assertion follows.
\end{proof}

Making the change of variables 
\begin{equation}m:=2k-1/2,\quad k \in \mathbb{N},
\end{equation}
\begin{equation}
y:=\frac{1-\sqrt{x}}{2}, \quad 0<y<1/2,
\end{equation}
one has
\begin{equation}\label{eq:reprphi}
\Phi_k(y)=-\pi\left({}_2F_1\left( m, 1-m,1;y\right)
+{}_2F_1\left( m, 1-m,1;1-y\right)\right).
\end{equation}

At the point $y=0$ the function ${}_2F_1\left( m, 1-m,1;y\right)$ is recessive and  ${}_2F_1\left( m, 1-m,1;1-y\right)$ is dominant. Therefore, further transformations are required to apply the Liouville-Green method to the second function. In particular, we show that ${}_2F_1\left( m, 1-m,1;1-y\right)$ has a similar shape with $\phi_k(x)$ studied in \cite{BF}. Note that the parameter $m$ is now half-integral.

\begin{lem}
Let $m:=2k-1/2$, $k \in \mathbb{N}$. Then
\begin{multline}
{}_2F_1\left( m, 1-m,1;1-y\right)=(-\log{y}+2\psi(1)-2\psi(m))\times \\
{}_2F_1\left( m, 1-m,1;y\right)+
\frac{1}{\pi}\left(\frac{\partial}{\partial a}+\frac{\partial}{\partial b}+2\frac{\partial}{\partial c}\right){}_2F_1\left( a, b, c;y\right)\Bigg{|}_{\substack{a=m\\b=1-m\\c=1}}.
\end{multline}
\end{lem}
\begin{proof}
By \cite[Eq.~33,~p.~107]{BE} we have
\begin{multline*}
{}_2F_1\left( m+u, 1-m+u,1;1-y\right)=\\
{}_2F_1\left( m+u, 1-m+u,1+2u;y\right)\frac{\Gamma(1)\Gamma(-2u)}{\Gamma(1-m-u)\Gamma(m-u)}+\\
{}_2F_1\left( m-u, 1-m-u,1-2u;y\right)\frac{\Gamma(1)\Gamma(2u)}{\Gamma(1-m+u)\Gamma(m+u)}y^{-2u}.
\end{multline*}
Computing the limit as $u \rightarrow 0$, we prove the assertion.
\end{proof}

\begin{lem}\label{lem:hyp}
For $m=2k-1/2$, $k \in \mathbb{N}$ one has
\begin{equation}\label{eq:hyp12}
{}_2F_1\left( m, 1-m,1;1/2\right)=-\frac{(-1)^k}{\sqrt{2\pi}}\frac{\Gamma(k-1/4)}{\Gamma(k+1/4)},
\end{equation}
\begin{equation}\label{eq:derhyp12}
\frac{d}{dx}\biggl({}_2F_1\left( m, 1-m,1;x\right)\biggr)\bigg|_{x=1/2}=\frac{4(-1)^k}{\sqrt{2\pi}}\frac{\Gamma(k+1/4)}{\Gamma(k-1/4)}.
\end{equation}
\end{lem}
\begin{proof}
On the one hand, by equation \eqref{eq:decompphi} we have
\begin{equation*}
 {}_2F_1\left( m, 1-m,1;1/2\right)=-\frac{1}{2\pi}\Phi_k(0).
\end{equation*}
On the other hand, equation \eqref{defphi} yields
\begin{equation*}
\Phi_k(0)=\frac{\Gamma(k-1/4)\Gamma(3/4-k)}{\Gamma(1/2)}=(-1)^k\sqrt{2\pi}\frac{\Gamma(k-1/4)}{\Gamma(k+1/4)}.
\end{equation*}
The last two equalities imply \eqref{eq:hyp12}.

As a consequence of \cite[Eq.~15.8.25]{HMF} we obtain
\begin{multline*}
 {}_2F_1\left( m, 1-m,1;x\right)=\frac{\Gamma(1/2)}{\Gamma(m/2+1/2)\Gamma(1-m/2)} \times\\
{}_2F_1\left( m/2, 1/2-m/2,1/2;(1-2x)^2\right)+\\
(1-2x)\frac{\Gamma(-1/2)}{\Gamma(m/2)\Gamma(1/2-m/2)} {}_2F_1\left( m/2+1/2, 1-m/2,3/2;(1-2x)^2\right).
\end{multline*}
Then equality \eqref{eq:derhyp12} follows by differentiating the last expression in $x$ and setting $x=1/2$.
\end{proof}

\begin{lem}\label{lem:diffur}
The functions $$ {}_2F_1\left( m, 1-m,1;y\right)\text{ and }{}_2F_1\left( m, 1-m,1;1-y\right)$$ are solutions of  differential equation
\begin{equation}\label{eq:diffur}
y(1-y)F''(y)+(1-2y)F'(y)+m(m-1)F(y)=0.
\end{equation}
\end{lem}
\begin{proof}
This follows from the differential equation for hypergeometric functions.
\end{proof}

\subsection{Approximation of $\Phi_k$}

In order to find a Liouville-Green approximation for $\Phi_k(y)$ we use formula \eqref{eq:reprphi} and study separately each of the hypergeometric functions $${}_2F_1\left( m, 1-m,1;y\right)\text{
 and }{}_2F_1\left( m, 1-m,1;1-y\right).$$ 
As shown in Lemma \ref{lem:diffur}, these functions are solutions of differential equation \eqref{eq:diffur}
that was already approximated in \cite[Section~5.2]{BF}. So our problem reduces to computation of the Liouville-Green constants $C_Y$ and $C_J$ in the approximation of  ${}_2F_1\left( m, 1-m,1;1-y\right).$

For the reader's convenience, we  briefly recall the required results of \cite[Section~5.2]{BF}.
It follows from Lemma \ref{lem:diffur} that the functions 
\begin{equation}G_1(y):={}_2F_1\left( m, 1-m,1;y\right)\sqrt{y(1-y)},
\end{equation} 
\begin{equation}G_2(y):={}_2F_1\left( m, 1-m,1;1-y\right)\sqrt{y(1-y)}
\end{equation} 
are solutions of differential equation
\begin{equation}\label{diffurufg}
G''(y)=(u^2f(y)+g(y))G(y),
\end{equation}
where
\begin{equation}
u:=2k-1,\quad f(y):=-\frac{1}{y(1-y)},
\end{equation}
\begin{equation}
g(y):=-\frac{1}{4y^2(1-y)^2}+\frac{1}{4y(1-y)}.
\end{equation}
Making the change of variables
\begin{equation}\label{req2}
Z(y):=\frac{G(y)}{\alpha(y)}, \quad \alpha(y):=\frac{(y-y^2)^{1/4}}{2(\arcsin{\sqrt{y}})^{1/2}},
\end{equation}
\begin{equation}
\xi:=4\arcsin^2{\sqrt{y}},
\end{equation}
we transform equation  \eqref{diffurufg} into the following shape
\begin{equation}\label{diffurzxi}
\frac{d^2Z}{d \xi^2}+\left[\frac{u^2}{4\xi}+\frac{1}{4\xi^2}+\frac{\psi(\xi)}{\xi} \right]Z=0
\end{equation}
with
\begin{equation}
\psi(\xi):=\frac{1}{16 \sin^2{\sqrt{\xi}}}-\frac{1}{16\xi}.
\end{equation}

Removing the summand with $\psi(\xi)/\xi$ in equation \eqref{diffurzxi}, we have
\begin{equation}\label{diffurjy}
\frac{d^2Z}{d \xi^2}+\left[\frac{u^2}{4\xi}+\frac{1}{4\xi^2} \right]Z=0.
\end{equation}
The solutions of \eqref{diffurjy} are defined by
\begin{equation}
Z_C=\sqrt{\xi}C_0(u\sqrt{\xi}),
\end{equation}
where $C_i$ is either $J$ or $Y$ Bessel function of index $i$.

Then according to  \cite[Chapter~12]{O}  solutions of original differential equation \eqref{diffurzxi} can be found in the form
\begin{equation}\label{solution}
Z_C(\xi)=\sqrt{\xi}C_0(u\sqrt{\xi})\sum_{n=0}^{\infty}\frac{A(n;\xi)}{u^{2n}}-\frac{\xi}{u}C_1(u\sqrt{\xi})\sum_{n=0}^{\infty}\frac{B(n;\xi)}{u^{2n}}.
\end{equation}
In order to determine coefficients $A(n;\xi)$, $B(n;\xi)$ we use differential equations (see \cite[Eq.~8.491(3)]{GR})
for functions
\begin{equation}
W(\xi):=\sqrt{\xi}C_0(u\sqrt{\xi}), \quad V(\xi):=\xi C_1(u\sqrt{\xi})
\end{equation}
and substitute \eqref{solution} in equation \eqref{diffurzxi}. This yields
\begin{equation}
W(\xi)\sum_{n=0}^{\infty}\frac{C_n(\xi)}{u^{2n}}-V(\xi)\sum_{n=0}^{\infty}\frac{D_n(\xi)}{u^{2n-1}}=0,
\end{equation}
where
\begin{equation*}
C_n(\xi):=A''(n;\xi)+\frac{1}{\xi}A'(n;\xi)-\frac{\psi(\xi)}{\xi}A(n;\xi)-B'(n;\xi)-\frac{B(n;\xi)}{2\xi},
\end{equation*}
\begin{equation*}
D_n(\xi):=B''(n-1;\xi)+\frac{1}{\xi}B'(n-1;\xi)-\frac{\psi(\xi)}{\xi}B(n-1;\xi)+\frac{1}{\xi}A'(n;\xi).
\end{equation*}

Letting $C_n(\xi)=D_n(\xi)=0$, we find the required recurrence relations
\begin{equation}\label{recurrence1}
A(n;\xi)=-\xi B'(n-1;\xi)+\int_{0}^{\xi}\psi(x)B(n-1;x)dx+\lambda_n,
\end{equation}
\begin{equation}\label{recurrence2}
\sqrt{\xi}B(n;\xi)=\int_{0}^{\xi}\frac{1}{\sqrt{x}}\left(xA''(n;x)+A'(n;x)-\psi(x)A(n;x) \right)dx
\end{equation}
for some real constants of integration $\lambda_n$.

Assume that $A(0;\xi)=1$. Then
\begin{equation}\label{eq:bo}
B(0;\xi)=-\frac{1}{8\sqrt{\xi}}\left(\cot{\sqrt{\xi}}-\frac{1}{\sqrt{\xi}}\right),
\end{equation}
\begin{multline}\label{eq:a1}
A(1;\xi)=\frac{1}{8}\left(\frac{1}{\xi}-\frac{\cot{\sqrt{\xi}}}{2\sqrt{\xi}}-\frac{1}{2\sin^2{\sqrt{\xi}}} \right)\\
-\frac{1}{128}\left(\cot{\sqrt{\xi}}-\frac{1}{\sqrt{\xi}} \right)^2+\lambda_1.
\end{multline}

Furthermore, solutions \eqref{solution} can be approximated by finite series using \cite[Theorem~4.1,~p.~444]{O} or  \cite[Theorem~1]{BD}. 
\begin{thm}\label{LGphi}
Let $\xi_2=\pi^2/4$. For each value of $u$ and each nonnegative integer $N$ equation \eqref{diffurzxi} has solutions $Z_Y(\xi)$, $Z_J(\xi)$ which are infinitely differentiable in $\xi$ on interval $(0, \xi_2)$, and are given by
\begin{multline}\label{zyxi}
Z_Y(\xi)=\sqrt{\xi}Y_0(u\sqrt{\xi})\sum_{n=0}^{N}\frac{A_Y(n;\xi)}{u^{2n}}-\\
\frac{\xi}{u}Y_1(u\sqrt{\xi})\sum_{n=0}^{N-1}\frac{B_Y(n;\xi)}{u^{2n}}+\epsilon_{2N+1,1}(u,\xi),
\end{multline}
\begin{multline}\label{eq:zjapprox}
Z_J(\xi)=\sqrt{\xi}J_0(u\sqrt{\xi})\sum_{n=0}^{N}\frac{A_J(n;\xi)}{u^{2n}}-\\
\frac{\xi}{u}J_1(u\sqrt{\xi})\sum_{n=0}^{N-1}\frac{B_J(n;\xi)}{u^{2n}}+\epsilon_{2N+1,2}(u,\xi),
\end{multline}
where
\begin{equation}\label{zyxi2}
\epsilon_{2N+1,1}(u,\xi)\ll \frac{\sqrt{\xi}|Y_0(u\sqrt{\xi})|}{u^{2N+1}}\sqrt{\xi_2-\xi},
\end{equation}
\begin{equation}
\epsilon_{2N+1,2}(u,\xi)\ll \frac{\sqrt{\xi}|J_0(u\sqrt{\xi})|}{u^{2N+1}}\min{(\sqrt{\xi},1)}
\end{equation}
and coefficients $(A_Y(n;\xi),B_Y(n;\xi))$, $(A_J(n;\xi),B_J(n;\xi))$ are defined by \eqref{recurrence1}-\eqref{recurrence2}.

\end{thm}

 Functions $\xi^{1/4}(\sin{\sqrt{\xi}})^{1/2}G_1(\sin^2{\sqrt{\xi}/2})$ and $Z_J(\xi)$ are recessive solutions of equation \eqref{diffurzxi}
 as $\xi \rightarrow 0$. Therefore,  there is $c_0$ such that
\begin{equation}\label{f1coeff}
\xi^{1/4}(\sin{\sqrt{\xi}})^{1/2}G_1(\sin^2{\sqrt{\xi}/2})=c_0Z_J(\xi).
\end{equation}
The value of constant $c_0$ is determined by computing the limit of the left and right-hand sides  of equation
\eqref{f1coeff} as $\xi \rightarrow 0$.
On the one hand,
\begin{equation}
\lim_{\xi \rightarrow 0}{}_2F_{1}(k,1-k,1;\sin^2{\sqrt{\xi}/2})=1.
\end{equation}
On the other hand,
\begin{equation}
Z_J(\xi)=\sqrt{\xi}\sum_{n=0}^{N}\frac{A_J(n;\xi)}{u^{2n}}+O(\xi) \text{ as } \xi \rightarrow 0.
\end{equation}
Choosing $A_J(n;\xi)$ such that $A_J(0;0)=1$ and $A_J(n;0)=0$ for $n \geq 1$ we find
that $c_0=1$.

To sum up, we proved the following Lemma.
\begin{lem}\label{lem:g1}
Let $\xi_2=\pi^2/4$. For $\xi \in (0,\xi_2)$ one has
\begin{equation}
\xi^{1/4}(\sin\sqrt{\xi})^{1/2}G_1\left( \sin^2{\frac{\sqrt{\xi}}{2}}\right)=Z_J(\xi),
\end{equation}
where $Z_J(\xi)$ is given by \eqref{eq:zjapprox}.
\end{lem}

This concludes the summary of results of \cite[Section~5.2]{BF}. Our final goal is to
compute $C_Y=C_Y(u)$ and $C_J=C_J(u)$ such that
\begin{equation}\label{eq:cycj2}
\xi^{1/4}(\sin{\sqrt{\xi}})^{1/2}G_2(\sin^2{\sqrt{\xi}/2})=C_YZ_Y(\xi)+C_JZ_J(\xi).
\end{equation}

Note that there exist $c_1,c_2$ such that
\begin{multline}\label{eq:zyxi}
Z_Y(\xi)=\left(c_1G_1\left( \sin^2{\frac{\sqrt{\xi}}{2}}\right)+c_2G_2\left( \sin^2{\frac{\sqrt{\xi}}{2}}\right) \right)\times \\ \xi^{1/4}(\sin\sqrt{\xi})^{1/2} =
c_1Z_J(\xi)+\xi^{1/4}(\sin\sqrt{\xi})^{1/2}c_2G_2\left( \sin^2{\frac{\sqrt{\xi}}{2}}\right).
\end{multline}
The last two equalities imply
\begin{equation}\label{eq:CYCJ}
C_Y=\frac{1}{c_2},\quad C_J=-\frac{c_1}{c_2}.
\end{equation}

\begin{lem}\label{lem:c1c2}
One has
\begin{multline}
2c_1=(-1)^{k+1}\sqrt{2\pi}\frac{\Gamma(k+1/4)}{\Gamma(k-1/4)}\frac{Z_Y(\xi_2)}{\xi_{2}^{1/4}}+(-1)^k\sqrt{2\pi}\times \\\frac{\Gamma(k-1/4)}{\Gamma(k+1/4)} \xi_{2}^{1/4}\left(Z_Y'(\xi_2)-\frac{Z_Y(\xi_2)}{4\xi_2} \right),
\end{multline}
\begin{multline}
2c_2=(-1)^{k+1}\sqrt{2\pi}\frac{\Gamma(k+1/4)}{\Gamma(k-1/4)}\frac{Z_Y(\xi_2)}{\xi_{2}^{1/4}}-(-1)^k\sqrt{2\pi}\times \\\frac{\Gamma(k-1/4)}{\Gamma(k+1/4)} \xi_{2}^{1/4}\left(Z_Y'(\xi_2)-\frac{Z_Y(\xi_2)}{4\xi_2} \right).
\end{multline}
\end{lem}
\begin{proof}
To determine coefficients $c_1,c_2$ we consider the pair of equations
\begin{equation*}
Z_Y(\xi_2)=\xi_{2}^{1/4}\left( c_1G_1(1/2)+c_2G_2(1/2)\right),
\end{equation*}
\begin{equation*}
Z_Y'(\xi_2)=\frac{Z_Y(\xi_2)}{4\xi_2}+\frac{1}{4\xi_{2}^{1/4}}\left( c_1G_1'(1/2)+c_2G_2'(1/2)\right).
\end{equation*}
Note that $$G_1(1/2)=G_2(1/2)\text{ and }G_1'(1/2)=-G_2'(1/2).$$
Therefore,
\begin{equation*}
(c_1+c_2)G_1(1/2)=\frac{Z_Y(\xi_2)}{\xi_{2}^{1/4}},
\end{equation*}
\begin{equation*}
(c_1-c_2)G_1'(1/2)=4\xi_{2}^{1/4}\left(Z_Y'(\xi_2)-\frac{Z_Y(\xi_2)}{4\xi_{2}}\right).
\end{equation*}
The assertion follows by Lemma \ref{lem:hyp}.
\end{proof}

\begin{lem}\label{lem:zyzyd} For $\xi_2=\pi^2/4$ one has
\begin{equation}\label{eq:zyxi2}
Z_Y(\xi_2)=\frac{(-1)^{k+1}}{\sqrt{2u}}\left[1+\frac{1}{u^2}\left( \lambda_1-\frac{1}{16}\right)+O\left( \frac{1}{u^3}\right)\right],
\end{equation}
\begin{equation}\label{eq:zyxi2d}
Z_{Y}'(\xi_2)=\frac{(-1)^{k+1}}{\sqrt{2u}}\left[\frac{u}{\pi}+\frac{1}{\pi^2}+\frac{2\lambda_1+1/8}{2\pi u}
+O\left(\frac{1}{u^2} \right)\right].
\end{equation}
\end{lem}
\begin{proof}

Applying \cite[Theorem~1]{BD} with $N=1$ we obtain
\begin{equation*}
\epsilon_{3,1}(u;\xi_2)=0 \text{ and }\frac{\partial}{\partial \xi}\epsilon_{3,1}(u;\xi)\bigg|_{\xi=\xi_2}=0.
\end{equation*}
Therefore,
\begin{equation*}
Z_Y(\xi_2)=\sqrt{\xi_2}Y_0(u\sqrt{\xi_2})\left(1+\frac{A_Y(1,\xi_2)}{u^2}\right)-
\xi_2Y_1(u\sqrt{\xi_2})\frac{B_Y(0,\xi_2)}{u}
\end{equation*}
and
\begin{multline*}
Z_Y'(\xi_2)=\sqrt{\xi_2}Y_0(u\sqrt{\xi_2})\biggl(\frac{1}{2\xi_2}\left[1+\frac{A_Y(1;\xi_2)}{u^{2}}
\right]+\frac{A_Y'(1;\xi_2)}{u^{2}}
-\\ \frac{1}{2}B_Y(0;\xi_2) \biggr)-
\xi_2Y_1(u\sqrt{\xi_2})\biggl( \frac{u}{2\xi_2}\left[1+\frac{A_Y(1;\xi_2)}{u^{2}}
\right]+\\\frac{1}{2\xi_2 u}B_Y(0;\xi_2)
+\frac{1}{u}B_Y'(0;\xi_2) \biggr).
\end{multline*}

By means of the Hankel asymptotic expansion (see \cite[Eq.~10.17.1,~10.17.4]{HMF} and \cite[Eq.~8.451(1,7,8)]{GR}) we evaluate 
\begin{multline*}
\sqrt{\xi_2}Y_0(u\sqrt{\xi_2})=\frac{\pi}{2}Y_0\left((2k-1)\frac{\pi}{2}\right)=\\
\frac{(-1)^{k+1}}{\sqrt{2u}}\left[\sum_{j=0}^{\infty}(-1)^j\frac{a_{2j}(0)}{(\pi u/2)^{2j}}
+\sum_{j=0}^{\infty}(-1)^j\frac{a_{2j+1}(0)}{(\pi u/2)^{2j+1}}\right],
\end{multline*}
\begin{equation*}
\xi_2Y_1(u\sqrt{\xi_2})=\frac{\pi}{2}\frac{(-1)^{k}}{\sqrt{2u}}\left[\sum_{j=0}^{\infty}(-1)^j\frac{a_{2j}(1)}{(\pi u/2)^{2j}}
-\sum_{j=0}^{\infty}(-1)^j\frac{a_{2j+1}(1)}{(\pi u/2)^{2j+1}}\right],
\end{equation*}
where \begin{equation*}
a_j(v)=\frac{\Gamma(v+j+1/2)}{2^j j! \Gamma(v-j+1/2)}.
\end{equation*}
This yields
\begin{multline*}
Z_Y(\xi_2)=\frac{(-1)^{k+1}}{\sqrt{2u}}\left[1+\frac{a_1(0)}{\pi u/2}-\frac{a_2(0)}{(\pi u/2)^2}+O(u^{-3})\right]
\left(1+\frac{A_Y(1;\xi_2)}{u^2} \right)-\\
\frac{\pi}{2}\frac{(-1)^k}{\sqrt{2u}}\left[1-\frac{a_1(1)}{\pi u/2}-\frac{a_2(1)}{(\pi u/2)^2}+O(u^{-3})\right]\frac{B_Y(0,\xi_2)}{u}.
\end{multline*}
Simplifying the expression above, one has
\begin{multline*}
Z_Y(\xi_2)=\frac{(-1)^{k+1}}{\sqrt{2u}}\biggl[1+\frac{1}{u}\left(\frac{2}{\pi}a_1(0)+\frac{\pi}{2}B_Y(0,\xi_2) \right)+\\ \frac{1}{u^2}\left(A_Y(1,\xi_2)-\frac{4}{\pi^2}a_2(0)-a_1(1)B_Y(0,\xi_2) \right)\biggr]+O(u^{-7/2}).
\end{multline*}
Using formulas \eqref{eq:bo} and  \eqref{eq:a1}, we find
\begin{equation*}
a_1(0)=-\frac{1}{8},\quad \frac{2}{\pi}a_1(0)+\frac{\pi}{2}B_Y(0,\xi_2)=0,
\end{equation*}
\begin{equation*}
A_Y(1,\xi_2)-\frac{4}{\pi^2}a_2(0)-a_1(1)B_Y(0,\xi_2)=\lambda_1-\frac{1}{16}.
\end{equation*}
This gives equation \eqref{eq:zyxi2}.

Similarly, we obtain
\begin{multline*}
Z_Y'(\xi_2)=\sqrt{\xi_2}Y_0(u\sqrt{\xi_2})\left[\frac{7}{4\pi^2}+\frac{1}{u^2}\left( \frac{1}{\pi^2}\left(2\lambda_1-\frac{1}{32}\right)-\frac{15}{8\pi^4}\right)\right]-\\
\xi_2Y_1(u\sqrt{\xi_2})\left[u\frac{2}{\pi^2}+\frac{1}{u}\left(\frac{1}{\pi^2}\left( 2\lambda_1+\frac{1}{8}\right)-\frac{1}{16\pi^4} \right)\right].
\end{multline*}
Hankel's expansion for Bessel functions yields the following asymptotics
\begin{multline*}
Z_Y'(\xi_2)=\frac{(-1)^{k+1}}{\sqrt{2u}}\biggl[u\frac{1}{\pi}+\left(\frac{7}{4\pi^2}-\frac{2}{\pi^2}a_1(1)\right)+\\ \frac{1}{u}\left( \frac{7}{2\pi^3}a_1(0)+\frac{\pi}{2}\left(\frac{2\lambda_1+1/8}{\pi^2}-\frac{1}{16\pi^4} \right)-\frac{4}{\pi^3}a_2(1)\right)+\\
\frac{1}{u^2}\left( -\frac{7}{\pi^4}a_2(0)-a_1(1)\left( \frac{2\lambda_1+1/8}{\pi^2}-\frac{1}{16\pi^4}\right)+\frac{8}{\pi^4}a_3(1)\right)+O(u^{-3}) \biggr].
\end{multline*}
Finally, substituting
\begin{equation*}
a_1(0)=-\frac{1}{8}, \quad a_1(1)=\frac{3}{8}, \quad a_2(1)=-\frac{15}{128}
\end{equation*}
we prove equation \eqref{eq:zyxi2d}.
\end{proof}

\begin{cor}\label{corcjcy}
One has
\begin{equation}
C_Y=1+O\left(\frac{1}{k}\right), \quad C_J=O\left(\frac{1}{k^2} \right).
\end{equation}
\end{cor}
\begin{proof}
By Lemma \ref{lem:zyzyd}
\begin{equation*}
Z_Y'(\xi_2)-\frac{Z_Y(\xi_2)}{4\xi_2}=\frac{(-1)^{k+1}}{\sqrt{2u}}\left(\frac{u}{\pi}+\frac{2\lambda_1+1/8}{2\pi u}+O(u^{-2})\right).
\end{equation*}

It follows from  \cite[Eq.~5.11.13]{HMF} that
\begin{equation*}
\frac{\Gamma(k+1/4)}{\Gamma(k-1/4)}=k^{1/2}-\frac{1}{4}k^{-1/2}+O(k^{-3/2})
\end{equation*}
and 
\begin{equation*}
\frac{\Gamma(k-1/4)}{\Gamma(k+1/4)}=k^{-1/2}+O(k^{-3/2}).
\end{equation*}
Then Lemma \ref{lem:c1c2} gives
\begin{equation*}
c_1=O(k^{-2}), \quad c_2=1+O(k^{-1}).
\end{equation*}
Equations \eqref{eq:CYCJ} yield the assertion.
\end{proof}

As a consequence of equation \eqref{eq:cycj2}, Lemma \ref{lem:g1}, Theorem \ref{LGphi} and Corollary \ref{corcjcy} we obtain the main result.
\begin{thm}\label{thm:apprphi}
Let $u=2k-1$, $\xi_2=\pi^2/4$. Then for $\xi \in (0,\xi_2)$ one has
\begin{equation}
\Phi_k(\cos^2\sqrt{\xi})=\frac{-\pi}{\xi^{1/4}(\sin\sqrt{\xi})^{1/2}}\left[Z_J(\xi)+C_YZ_Y(\xi)+C_JZ_J(\xi)\right],
\end{equation}
where $Z_Y$, $Z_J$ are given by \eqref{zyxi}, \eqref{eq:zjapprox} and
\begin{equation}
C_Y=1+O\left(\frac{1}{k}\right), \quad C_J=O\left(\frac{1}{k^2} \right).
\end{equation}

\end{thm}

\subsection{Approximation of $\Psi_k$}
Next, we find a Liouville-Green approximation for the function $\Psi_k$.
With this goal, we follow the arguments of \cite[Section~5.3]{BF} with minor changes. In particular, the differential equation for $\Psi_k(x)$ is slightly different, and, therefore, one requires to recompute various functions and constants appearing in the Liouville-Green approximation. We provide all details here to make the presentation self-contained.

Consider the function \begin{equation}
y(x):=\sqrt{1-x}\Psi_k(x).
\end{equation}
Let
$u:=k-1/2$ and
\begin{equation}
f(x):=\frac{1}{x^2(1-x)}, \quad g(x):=-\frac{1}{4x^2(1-x)^2}+\frac{3}{16x(1-x)}.
\end{equation}

\begin{lem}
The function $y=y(x)$ is a solution of equation
\begin{equation}\label{eq:diffurf21}
y''(x)-(u^2f(x)+g(x))y(x)=0.
\end{equation}
\end{lem}
\begin{proof}
Using the differential equation for the hypergeometric function, we find that $y=y(x)$ satisfies the following differential equation
\begin{equation*}
y''+\biggl( \frac{1-(2k-1)^2}{4x^2} +\frac{1}{4(1-x)^2}+\frac{5/4-(2k-1)^2}{4x(1-x)}\biggr)y=0.
\end{equation*}
The assertion follows.
\end{proof}

Making the change
\begin{equation}
Z(x):=\frac{y(x)}{\alpha(x)}, \quad \alpha(x):=\frac{(x^2-x^3)^{1/4}}{2(\artanh{\sqrt{1-x}})^{1/2}}
\end{equation}
and the substitution
\begin{equation}
\xi:=4 \artanh^2{\sqrt{1-x}},
\end{equation}
we transform equation \eqref{eq:diffurf21} to the type
\begin{equation}\label{eq:diffurZpsi}
\frac{d^2Z}{d\xi^2}+\left[ -\frac{u^2}{4\xi}+\frac{1}{4\xi^2}-\frac{\psi(\xi)}{\xi}\right]Z=0,
\end{equation}
where
\begin{equation}
\psi(\xi)=\frac{1}{16}\left(\frac{1}{\xi}-\frac{1}{4\sinh^2{\sqrt{\xi}/2}} \right)
\end{equation}
is an analytic function as $\xi \rightarrow 0$.

In order to find a Liouville-Green approximation to equation \eqref{eq:diffurZpsi}, we remove the term with $\psi(\xi)/\xi$ in \eqref{eq:diffurZpsi}. The resulting equation
\begin{equation}
Z''+\left(-\frac{u^2}{4\xi}+\frac{1}{4\xi} \right)Z=0
\end{equation}
has  $I$ and $K$ Bessel functions as solutions (see \cite[Eq.~10.13.2]{HMF}), namely 
\begin{equation}
Z_L=\sqrt{\xi}L_0(u\sqrt{\xi}),
\end{equation}
where 
\begin{equation}
L_v(x):=\begin{cases}
I_v(x)\\e^{\pi i v}K_v(x)
\end{cases}.
\end{equation}
This suggests that solutions of the original differential equation \eqref{eq:diffurZpsi} can be written in the form (see \cite[Eq.~2.09,~Chapter~12]{O}) 
\begin{equation}\label{eq:solZ}
Z_L=\sqrt{\xi}L_0(u\sqrt{\xi})\sum_{n=0}^{\infty}\frac{A(n;\xi)}{u^{2n}}+\frac{\xi}{u}L_1(u\sqrt{\xi})
\sum_{n=0}^{\infty}\frac{B(n;\xi)}{u^{2n}}.
\end{equation}

\begin{lem}
Coefficients $A(n;\xi)$ and $B(n;\xi)$ are given by
\begin{multline}\label{rec:bnxi}
\sqrt{\xi}B(n;\xi)=-\sqrt{\xi}A'(n;\xi)+\\\int_{0}^{\xi}\left(\psi(x) A(n;x)-\frac{1}{2}A'(n;x)\right)\frac{dx}{\sqrt{x}},
\end{multline}
\begin{equation}
A(n;\xi)=-\xi B'(n-1;\xi)+\int_{0}^{\xi}\psi(x)B(n-1;x)dx+\lambda_n
\end{equation}
for some real constants of integration $\lambda_n$.
\end{lem}
\begin{proof}
By \cite[Eq.~10.13.2,~10.13.5,~10.36,~10.29.2,~10.29.3]{HMF} the functions
\begin{equation*}
W(\xi):=\sqrt{\xi}L_0(u\sqrt{\xi}),\quad V(\xi):=\xi L_1(u\sqrt{\xi})
\end{equation*}
 satisfy the following relations 
\begin{equation*}
W''+\left( -\frac{u^2}{4\xi}+\frac{1}{4\xi^2}\right)W=0,
\end{equation*}
\begin{equation*}
V''-\frac{1}{\xi}V'+\left( -\frac{u^2}{4\xi}+\frac{3}{4\xi^2}\right)V=0,
\end{equation*}
\begin{equation*}
V'=\frac{1}{2\xi}V+\frac{u}{2}W, \quad W'=\frac{1}{2\xi}W+\frac{u}{2\xi}V.
\end{equation*}

Using this and substituting solution \eqref{eq:solZ} into equation \eqref{eq:diffurZpsi}, we obtain that
\begin{equation*}
W(\xi)\sum_{n=0}^{\infty}\frac{C(n;\xi)}{u^{2n}}+V(\xi)\sum_{n=0}^{\infty}\frac{D(n;\xi)}{u^{2n+1}}=0,
\end{equation*}
where
\begin{equation*}
C(n;\xi)=A''(n;\xi)+\frac{A'(n;\xi)}{\xi}-\frac{\psi(\xi)}{\xi}A(n;\xi)+B'(n;\xi)+\frac{B(n;\xi)}{2\xi},
\end{equation*}
\begin{equation*}
D(n;\xi)=B''(n-1;\xi)+\frac{B'(n-1;\xi)}{\xi}-\frac{\psi(\xi)}{\xi}B(n-1;\xi)+\frac{A'(n;\xi)}{\xi}.
\end{equation*}

Setting $C(n;\xi)=D(n;\xi)=0$ we find the required recurrence relations.
\end{proof}

Let $A(0;\xi)=1$. Then
\begin{equation}
B(0;\xi)=\frac{1}{16}\left( \frac{\coth{\sqrt{\xi}}}{\sqrt{\xi}}-\frac{2}{\xi}\right),
\end{equation}

\begin{multline}
A(1;\xi)=-\frac{1}{32}\left( \frac{4}{\xi}-\frac{\coth{\sqrt{\xi/4}}}{\sqrt{\xi}}-\frac{1}{2\sinh^2{\sqrt{\xi/4}}}\right)+\\
\frac{1}{512}\left( \coth{\sqrt{\xi/4}}-\frac{2}{\sqrt{\xi}}\right)^2+\lambda_1.
\end{multline}

Note that
\begin{equation}\label{eq:baxi0}
\lim_{\xi \rightarrow \infty}\sqrt{\xi}B(0;\xi)=\frac{1}{16}, \quad \lim_{\xi \rightarrow \infty}A(1;\xi)=\frac{1}{512}+\lambda_1.
\end{equation}

The variation of the function is given by
\begin{equation}
V_{a,b}(f(x)):=\int_{a}^{b}|(f(x))'|dx.
\end{equation}
\begin{lem}\label{lem:variation}
The function $V_{\xi,\infty}(\sqrt{x}B(1;x))$
is bounded. For $n>1$ the function
$V_{\xi,\infty}(\sqrt{x}B(n;x))$
converges.
\end{lem}
\begin{proof}
As a consequence of recurrence relation \eqref{rec:bnxi}, we find
\begin{equation*}
(\sqrt{x}B(1;x))'=O(x^{-1/2}) \text{ as } x \rightarrow 0
\end{equation*}
and
\begin{equation*}
(\sqrt{x}B(1;x))'=O(x^{-2}) \text{ as } x \rightarrow \infty.
\end{equation*}
Thus $V_{\xi,\infty}(\sqrt{x}B(1;x))$ is bounded.
Note that
\begin{equation*}
\psi^{(s)}(\xi)=O\left(\frac{1}{|\xi|^{s+1}} \right).
\end{equation*}
The convergence of variation for $n>1$ then follows by \cite[Exercise~4.2,~p.~445]{O}.
\end{proof}

Using Lemma \ref{lem:variation}, we can truncate the infinite summation in \eqref{eq:solZ} up to $N$ summands with a negligible error term. The value of $N$ determines the quality of approximation: the error is smaller for larger $N$.
 
\begin{lem}\label{thm:zk}
For each value of $u$ and each nonnegative integer $N$ equation \eqref{eq:diffurZpsi} has solution $Z_K(\xi)$ which is infinitely differentiable in $\xi$ on interval $(0, \infty)$ and is given by
\begin{multline}\label{eq:zkxi}
Z_K(\xi)=\sqrt{\xi}K_0(u\sqrt{\xi})\sum_{n=0}^{N}\frac{A_K(n;\xi)}{u^{2n}}-\\\frac{\xi}{u}K_1(u\sqrt{\xi})
\sum_{n=0}^{N-1}\frac{B_K(n;\xi)}{u^{2n}}+\epsilon_{2N+1,3}(u,\xi),
\end{multline}
where
\begin{multline}
|\epsilon_{2N+1,3}(u,\xi)|\leq  \frac{\sqrt{\xi}K_0(u\sqrt{\xi})}{u^{2N+1}}\times \\ V_{\xi,\infty}(\sqrt{\xi}B_K(N;\xi))exp\left( \frac{1}{u}V_{\xi, \infty}(\sqrt{\xi}B_K(0;\xi))\right).
\end{multline}
In particular, for $N=1$
\begin{equation}
\epsilon_{3,3}(u,\xi)\ll \frac{\sqrt{\xi}K_0(u\sqrt{\xi})}{u^{3}}\min\left(\sqrt{\xi}, \frac{1}{\xi}\right).
\end{equation}
\end{lem}

 As $\xi \rightarrow \infty$, differential equation \eqref{eq:diffurZpsi} has two recessive solutions, namely
\begin{equation}
Z(\xi)=
\Psi_k\left(\frac{1}{\cosh^2{\sqrt{\xi}/2}} \right)\left( \xi\sinh^2{\sqrt{\xi}}\right)^{1/4}
\end{equation}
and $Z_K(\xi)$ given by \eqref{eq:zkxi}. Thus there is $C_K=C_K(u)$ such that
\begin{equation}\label{eq:phiklim}
\Psi_k\left(\frac{1}{\cosh^2{\sqrt{\xi}/2}} \right)\left( \xi\sinh^2{\sqrt{\xi}}\right)^{1/4}=
C_KZ_K(\xi).
\end{equation}

The last step is to compute $C_K=C_K(u)$.
\begin{lem} One has
\begin{equation}\label{asymp:ck}
C_K=2+O(k^{-1}).
\end{equation}
\end{lem}
\begin{proof}
To determine $C_K$, we compute the limit  of the left and right- hand sides of equation \eqref{eq:phiklim} as $\xi \rightarrow \infty$. This implies
\begin{equation*}\label{eq:ckexpl}
C_K=\frac{\Gamma(k-1/4)\Gamma(k+1/4)}{\Gamma(2k)}\frac{2^{2k}\sqrt{u}}{\sqrt{\pi}}\left[\sum_{n=0}^{N}\frac{a_n}{u^{2n}}-\sum_{n=0}^{N-1}\frac{b_n}{u^{2n+1}}\right]^{-1},
\end{equation*}
where
\begin{equation*}
a_n=\lim_{\xi \rightarrow \infty}A(n;\xi), \quad b_n=\lim_{\xi \rightarrow \infty}B(n;\xi)\sqrt{\xi}.
\end{equation*}
According to \eqref{eq:baxi0} we know that
\begin{equation*}
a_0=1, \quad a_1=\frac{1}{512}+\lambda_1, \quad b_0=\frac{1}{16}.
\end{equation*}
Furthermore,
\begin{equation*}
\frac{\Gamma(k-1/4)\Gamma(k+1/4)}{ \Gamma(2k)}= \frac{\Gamma^2(k)}{\Gamma(2k)} (1+O(1/k))=\frac{2\sqrt{\pi}}{\sqrt{k}2^{2k}}(1+O(1/k)).
\end{equation*}
The assertion follows.
\end{proof}

Finally, we obtain the main Theorem.
\begin{thm}\label{thm:approxPsi}
For  $\xi \in (0, \infty)$ the following equality holds
\begin{equation}\label{eq:phiklim1}
\Psi_k\left(\frac{1}{\cosh^2{\sqrt{\xi}/2}} \right)\left( \xi\sinh^2{\sqrt{\xi}}\right)^{1/4}=
C_KZ_K(\xi),
\end{equation}
where $Z_K(\xi)$ is defined by \eqref{eq:zkxi} and $C_K=2+O(k^{-1})$.
\end{thm}

\section{Asymptotic formula}

Corollaries \ref{thm:asympformula2} and \ref{thm:asympformula} are derived from Theorem \ref{thm:explicitformula} by estimating the last two summands in exact formula \eqref{mainformula}, as we now show.

\begin{lem}
For any $\epsilon>0$ one has
\begin{equation}
E_1(k,l):=\frac{1}{\sqrt{l}}\sum_{1\leq n<2l}\mathscr{L}_{n^2-4l^2}(1/2)\Phi_k\left( \frac{n^2}{4l^2}\right)\ll
\frac{l^{5/6+\epsilon}}{\sqrt{k}}.
\end{equation}
\end{lem}
\begin{proof}
Using subconvexity bound \eqref{eq:subconvexity} we obtain
\begin{multline*}
E_1(k,l)\ll \frac{l^{\epsilon}}{\sqrt{l}}\sum_{1\leq n<2l}(2l-n)^{1/6}(2l+n)^{1/6}
\left|\Phi_k\left( \frac{n^2}{4l^2}\right)\right|\ll\\
\frac{l^{1/6+\epsilon}}{\sqrt{l}}\sum_{1\leq n<2l}n^{1/6}
\left|\Phi_k\left(\left(1- \frac{n}{2l}\right)^2\right)\right|
\end{multline*}

Let
\begin{equation*}
\xi=4\left(\arcsin{\sqrt{\frac{n}{4l}}} \right)^2, \quad u=2k-1.
\end{equation*}
Then by Theorem \ref{thm:apprphi} one has
\begin{multline*}
\Phi_k\left(\left(1- \frac{n}{2l}\right)^2\right)\ll \frac{\sqrt{\xi}Y_0(u\sqrt{\xi})}{(\arcsin(\sqrt{n/4l}))^{1/2}(n/l)^{1/4}}\ll\\
\frac{(\arcsin(\sqrt{n/4l}))^{1/2}}{(n/l)^{1/4}}Y_0(u\sqrt{\xi}).
\end{multline*}
If $l\ll k^2$, one has $u\sqrt{\xi}\gg 1$. Then the estimate for the Bessel function
\begin{equation*}
Y_0(u\sqrt{\xi})\ll \frac{1}{u^{1/2}\xi^{1/4}}
\end{equation*}
yields
\begin{equation*}
\Phi_k\left(\left(1- \frac{n}{2l}\right)^2\right)\ll\frac{1}{k^{1/2}(n/l)^{1/4}}.
\end{equation*}
Consequently,
\begin{equation*}
E_1(k,l)\ll \frac{l^{1/6+\epsilon}}{\sqrt{l}}\sum_{n<2l}\frac{n^{1/6}l^{1/4}}{k^{1/2}n^{1/4}}\ll
\frac{l^{5/6+\epsilon}}{\sqrt{k}}.
\end{equation*}
\end{proof}

\begin{cor}
If the Lindel\"{o}f hypothesis for Dirichlet L-functions is true, then for any $\epsilon>0$ one has
\begin{equation}
E_1(k,l)\ll \frac{l^{1/2+\epsilon}}{\sqrt{k}}.
\end{equation}

\end{cor}

\begin{lem}
For some $c>0$ one has
\begin{multline}
E_2(k,l):=\frac{1}{l\sqrt{2}}\sum_{n>2l}\mathscr{L}_{n^2-4l^2}(1/2)\sqrt{n}\Psi_k\left( \frac{4l^2}{n^2}\right)\ll\\
\frac{l^{-1/12}}{\sqrt{k}}\exp\left(-\frac{ck}{\sqrt{l}} \right).
\end{multline}
\end{lem}
\begin{proof}
It follows from subconvexity bound \eqref{eq:subconvexity} that
\begin{multline*}
E_2(k,l)\ll \frac{1}{l}\sum_{n>2l}(n^2-4l^2)^{1/6}\sqrt{n}\left|\Psi_k\left( \frac{4l^2}{n^2}\right)\right|\ll\\
\frac{1}{l}\int_{2l+1}^{\infty}x^{1/2}(x^2-4l^2)^{1/6}\left|\Psi_k\left( \frac{4l^2}{x^2}\right)\right|dx+
l^{-1/3}\left| \Psi_k\left( \frac{4l^2}{(2l+1)^2}\right)\right|.
\end{multline*}
Next, we make the change of variables $$x=2l\cosh{\frac{\sqrt{\xi}}{2}} $$  and estimate $E_2(k,l)$ using  Theorem \ref{thm:approxPsi} with $N=0$.
Consider the first summand 
\begin{multline*}
E_{2,1}(k,l):=\frac{1}{l}\int_{2l+1}^{\infty}x^{1/2}(x^2-4l^2)^{1/6}\left|\Psi_k\left( \frac{4l^2}{x^2}\right)\right|dx\ll
l^{5/6}\times \\\int_{\xi_0}^{\infty}\left(\sinh{\frac{\sqrt{\xi}}{2}}\right)^{5/6}\frac{|Z_k(\xi)|}{\xi^{3/4}}d\xi\ll 
l^{5/6}\int_{\xi_0}^{\infty}\left(\sinh{\frac{\sqrt{\xi}}{2}}\right)^{5/6}\frac{|K_0(u\sqrt{\xi})|}{\xi^{1/4}}d\xi,
\end{multline*}
where $u=k-1/2$ and the limit of integration $\xi_0$ is defined by $$\cosh{\frac{\sqrt{\xi_0}}{2}}=1+\frac{1}{2l}.$$

Making the change of variables
$\sqrt{\xi}=t$, one has $$t_0=4\arcsinh{\frac{1}{\sqrt{4l}}}.$$
Since $t_0 \gg 1/\sqrt{l}$ and $ut\geq ut_0\gg k/\sqrt{l}\gg 1$, we estimate the Bessel function as follows
\begin{equation*}
K_0(ut)\ll \frac{\exp(-ut)}{\sqrt{ut}}.
\end{equation*}
Finally,
\begin{multline*}
E_{2,1}(k,l)\ll 
l^{5/6}\int_{t_0}^{\infty}\left(\sinh{\frac{t}{2}}\right)^{5/6}\frac{|K_0(ut)|}{\sqrt{t}}tdt\ll\\
\frac{l^{5/6}}{\sqrt{k}}\int_{t_0}^{\infty}\left(\sinh{\frac{t}{2}}\right)^{5/6}e^{-ut}dt\ll
\frac{l^{5/6}}{\sqrt{k}}\times \\\left( \int_{t_0}^{1}t^{5/6}\exp{(-ut)}dt+\int_{1}^{\infty}\exp{(-ut+5t/12)}dt\right)
\ll \frac{l^{5/12}}{u\sqrt{k}}\exp(-ut_0).
\end{multline*}
The second summand can be estimated similarly:
\begin{multline*}
E_{2,2}(k,l):=l^{-1/3}\left| \Psi_k\left( \frac{4l^2}{(2l+1)^2}\right)\right|\ll
l^{-1/3}\frac{C_K|Z_k(\xi)|}{\xi^{1/4}(\sinh{\sqrt{\xi}})^{1/2}}\ll\\
l^{-1/3}l^{1/2}|Z_K(\xi)|\ll l^{1/6}\sqrt{\xi}|K_0(u\sqrt{\xi})|\ll \frac{\exp{(-ut_0)}}{k^{1/2}l^{1/12}}.
\end{multline*}
To sum up,
\begin{equation*}
E_{2}(k,l)\ll E_{2,1}(k,l)+E_{2,2}(k,l)\ll \frac{1}{l^{1/12}\sqrt{k}}\exp\left(\frac{-ck}{\sqrt{l}}\right)
\end{equation*}
for some $c>0$.
\end{proof}
\nocite{}

\end{document}